\documentclass{article}
\usepackage{macros}
\usepackage{hyperref}

\usepackage{chngcntr}
\counterwithout{equation}{section}

\usepackage{tikz}
\usepackage{setspace}
\def\b0{\boldsymbol{0}}
\def\bpsi{\boldsymbol{\psi}}

\def\bx{\boldsymbol{x}}
\def\bz{\boldsymbol{z}}
\def\by{\boldsymbol{y}}

\def\bv{\boldsymbol{v}}
\def\bu{\boldsymbol{u}}
\def\bb{{\boldsymbol{b}}}
\def\be{\boldsymbol{e}}
\def\bE{\boldsymbol{E}}
\def\bA{\boldsymbol{A}}

\def\bI{\boldsymbol{I}}
\bibliography{references}
\begin{document}
\title{You Need to Calm Down: Calmness Regularity for a Class of Seminorm Optimization Problems\thanks{This work was supported by NSF award DMS-18-30418.}
}
\date{\today}
\author{Alex Gutierrez, Gilad Lerman, Sam Stewart}
\maketitle
\begin{abstract}
Compressed sensing involves solving a minimization problem with objective function $\Omega(\bx) = \norm{\bx}_1$ and linear constraints $\bA \bx = \bb$. Previous work has explored robustness to errors in $\bA$ and $\bb$ under special assumptions. Motivated by these results, we explore robustness to errors in $\bA$ for a wider class of objective functions $\Omega$ and for a more general setting, where the solution may not be unique. Similar results for errors in $\bb$ are known and easier to prove. More precisely, for a seminorm $\Omega(\bx)$ with a polyhedral unit ball, we prove that the set-valued map $S(\bA) = \argmin_{\bA \bx = \bb} \Omega(\bx)$ is calm in $\bA$ whenever $\bA$ has full rank and the minimum value is positive, where calmness is a kind of local Lipschitz regularity.
\end{abstract}

\section{Introduction}
Convex optimization problems with linear constraints $\bA \bx = \bb$, where $\bA \in \R^{m \times n}$ is a full-rank matrix with $m < n$,
covers a wide variety of applications. A well-known example is compressed sensing. In this example, one aims to recover a sparse signal given a vector $\bb$ of reduced number of linear measurements and the very special measurement vectors stored as the rows of $\bA$. Under some restricting assumptions, the unknown signal is the solution of
$\bA \bx= \bb$ with minimal $\ell_1$ norm
\cite{donoho_elad2003,Candes_romberg_tao2006,Candes2006,donoho_cpam_2006, donoho2006compressed}, or equivalently, the solution of
\begin{equation}\label{compressed_sensing_problem}
	\min_{\bA \bx = \bb} \norm{\bx}_1.
\end{equation}
A particular case of interest, which is motivated by application to MRI, assumes Fourier measurements \cite{Candes_romberg_tao2006}. In this case, sampling can be done by taking $m=\mathcal{O}(s \log n)$ random measurements of the Fourier coefficients of the unknown signal $\bx^*$, where $s$ is an upper bound for the sparseness of $\bx^*$. These random measurements are dot products of the form $\innerp{\bx^*}{\bpsi_k}$, where $\bpsi_k$ is a randomly selected row of the $n \times n$ discrete Fourier transform matrix. Therefore, $\bA$ has rows $\bpsi_1$, $\ldots$, $\bpsi_k$ and $\bb = \bA \bx^*$.
Cand\`{e}s, Romberg and Tao~\cite{Candes_romberg_tao2006} showed that with probability $1$ the solution of \eqref{compressed_sensing_problem} is unique and recovers $\bx^*$.
This observation extends to other kinds of measurements (see \cite{donoho_cpam_2006} and also \cite{Candes2006, donoho2006compressed}).

In applications, the actual measurements of $\Ab$ and $\bb$ are noisy, so understanding the reconstruction error under noise is important.
Several works \cite{candes2006near,candes2008restricted,donoho_elad_temlyakov2006, tropp_relax_2006} have proven that for very special measurement matrices $\bA$ and for noisy measurements $\hat{\bb}$ of the form $\hat{\bb}=\bA \bx^* + \be$, where $\norm{\be}_2$ is sufficiently small and $\bx^*$ has a sufficiently close approximation in $\ell_1$ by an $s$-sparse vector, the solution $\bx$ of \eqref{compressed_sensing_problem} with $\bb$ replaced by $\hat{\bb}$ is sufficiently close to $\bx^*$ in $\ell_2$ norm. Herman and Strohmer \cite{Herman2010} extend this theory by allowing possible errors in the measurement matrix $\bA$.

Many other applications involve error in the matrix $\bA$. Examples include radar \cite{herman2009high}, remote sensing \cite{fannjiang2010compressed}, telecommunications \cite{gribonval2008atoms} and source separation \cite{blumensath2007compressed}. More recently, Gutierrez et al.~\cite{Gutierrez2019a} developed a compression technique for model-based MRI reconstruction. Here there are at least two sources of error: the discretization of a set of ODEs and the search for a sparse basis to represent the simulation results. They use the total variation norm instead of the $\ell^1$ norm.

Motivated by these broad model problems we assume a seminorm $\Omega$ from a particular class which includes the $\ell^1$ and total variation norms, and we prove a general stability result, with respect to perturbations in $\bA$, for the convex optimization problem
\begin{equation}\label{eq:our_optimization_problem}
	\min_{\bA \bx = \bb} \Omega(\bx).
\end{equation}
Even for the special case of the $\ell_1$ norm, our result differs from the earlier perturbation result of Herman and Strohmer \cite{Herman2010} since we do not enforce the more restricted regime of singleton solution sets. We are thus able to eliminate their conditions on the matrix ${\bA}$ and prove a kind of local stability result. For general convex $\Omega$, Klatte and Kummer \cite{Klatte2015} have shown that \eqref{eq:our_optimization_problem} is stable with respect to perturbations in $\bb$. We extend their result to stability in $\bA$ for a restricted class of $\Omega$.

The paper is organized as follows. In Section~\ref{sec:related_work} we define the various notions of set-valued regularity and describe stability results for set-valued functions that are related to our problem. In Section \ref{sec:proof},
we formulate and prove our theorem using tools from subspace perturbation theory, the theory of polyhedral mappings, the theory of local error bounds for convex inequality systems, and the theory of set-valued regularity.
Finally, in Section~\ref{section:future_work} we conclude this work and clarify the difficulty of extending our approach to the full class of seminorms.

\section{Background}\label{sec:related_work}

Our theory and proof combine ideas from different areas. We review here the necessary background according to the designated topics indicated by the section titles. We remark that the theory we use is formulated for a more general settings. Therefore, here and throughout the paper we formulate ideas both for our special setting and a setting of Banach spaces with a more general form of $\Omega$ and constraints.
We often use the same notation for both settings.

\subsection{Mathematical Formulation of Our Objective}
We consider the following solution map
\begin{equation}
\label{eq:def_S}
    S(\bA) = \argmin_{\bA \bx = \bb} \Omega(\bx)
    = \argmin_{\bx \in M(\bA)} \Omega(\bx),
\end{equation}
where $M(\bA)$ denotes the constraint set
\begin{equation*}
M(\bA) = \{ \bz \in \R^n \mid \bA \bz = \bb \}.
\end{equation*}
We note that generally $S(\bA)$ is a set-valued function.
We arbitrarily fix $\bA_0 \in \R^{m \times n}$ with full rank and quantitatively study the effect of perturbing $\bA_0$ on the set-valued solution map $S(\bA)$.

\subsection{Regularity of Set-Valued Functions}

In order to quantify the effect of perturbing $\bA_0$ on the set-valued solution map $S(\bA)$, we use some well-known notions of regularity of set-valued functions. We first motivate them for our setting by considering the special case where $\Omega(\bx) = \norm{\bx}_1$ and
$S(\bA_0)$ is an entire face of the $\ell^1$ ball. Then a small perturbation in $\bA_0$ results in a jump from that face to a single vertex as we demonstrate in Figure~\ref{fig:ex1}.
A model set-valued function with such a jump,
which needs to be handled by the developed theory,
is
\[
	G(t) =
	\begin{cases}
		[0, 1] & \textrm{ if } t  \in (1/2, 1] \cr	
		0 & \textrm{ if } t \in [0, 1/2).
	\end{cases}
\]
This function has several properties, which we will formulate below.

We will state their definitions in terms of Banach spaces $X$ and $Y$, where we denote by $2^X$ all subsets of $X$, $\|\cdot \|$ the norm of $Y$ and $d(\cdot,\cdot)$, the induced distance by $\|\cdot \|$ between points or sets in $X$. Throughout the paper we go back and forth between the general setting and our special setting of $X=\R^n$ with the Euclidean norm, $\|\cdot\|$, and $Y=\R^{m \times n}$ with the distance induced by the spectral norm. We also denote the spectral norm by $\|\cdot\|$. The induced distance by either norm is denoted by either $\|\cdot-\cdot\|$, when involving only matrices, or $d(\cdot, \cdot)$, when involving sets of matrices. Even though we use the same notation, e.g., $\|\cdot\|$ and $d(\cdot,\cdot)$ for different setting, they are easily distinguishable. When assuming norms and distances in the Euclidean case, we use vectors, which we denote by boldfaced small letters; when assuming matrices, we denote them by boldfaced capital letters; and when assuming abstract Banach spaces, we denote their elements by regular small letters.

It is somewhat intuitive that the above function $G(t)$ is ``upper semicontinuous''. We clarify this notion as follows.
\begin{definition}
A map $F : Y \to 2^X$ is upper semicontinuous at $y_0$ if for any open set $V$ intersecting $F(y_0)$ there exists a neighborhood $U$ of $y_0$ such that
\[
	F(y) \cap V \neq \emptyset, \quad \textrm{ for all } y \in U.
\]
\end{definition}

\begin{figure}[htbp]
\centering
\includegraphics[width=0.45\textwidth]{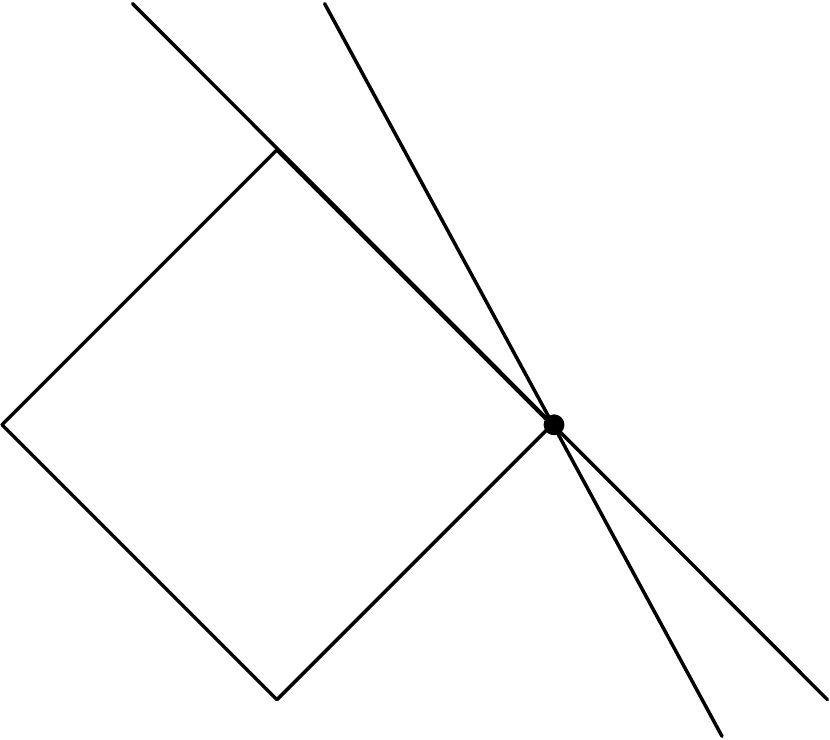}
\caption{Demonstration of a case where $S(\bA_0)$ is a face of the unit $\ell_1$ ball and $S(\bA)$, where $\bA=\bA_0+\bE$ and $0<\|\bE\| \ll 1$, is a vertex of this ball. Here, the line intersecting the face represents the solution of $\bA_0 \bx = \bb$, the other line represents the solution of $\bA \bx = \bb$, and the sets $S(\bA_0)$ and $S(\bA)$ are the intersections of these lines, respectively, with the unit $\ell_1$ ball, which is represented by the rotated square.}
\label{fig:ex1}
\end{figure}

It has been known since the 1970s that if $\Omega$ is convex, then the solution map $S(\bA)$ is upper semicontinuous  \cite[Theorem 1.15]{DiethardKlatte2002}.
Our result upgrades the upper semicontinuity of $S(\bA)$ to a kind of set-valued Lipschitz regularity, which we define next.
\begin{definition}
\label{def:calm}
A map $F : Y \to 2^X$ is calm at $(y_0, x_0) \in \textrm{Graph}(F)$ if there exist neighborhoods $U$ and $V$ of $y_0$ and $x_0$, respectively, and a constant $L(y_0, x_0)$ such that for all $y_1 \in U$ and $x_1 \in F(y_1) \cap V$,
\[
	d(x_1, F(y_0)) \leq L(y_0, x_0) \norm{ y_0 - y_1}.
\]
\end{definition}

We note that if $S(\bA)$ is calm at $(\bA_0, \bx_0)$ then for all $\bA_1$ sufficiently close, there exists $\bx_1 \in S(\bA_1)$ such that
\[
	\norm{\bx_0 - \bx_1} \leq C(\bA_0) \norm{\bA_0 - \bA_1}.
\]
This means that small perturbations to the measurement matrix $\bA$ leave at least two solutions $\bx_0 \in S(\bA_0)$ and $\bx_1 \in S(\bA_1)$ close.

We demonstrate that calmness is stronger than upper semicontinuity with the following example
\[
	H(t) = [-\sqrt{t},\sqrt{t}] \ \text{ for } t \in [0, 1].
\]
The upper semicontinuity of $H$ is obvious. It is not calm since for $t_0$ and $t_1$ with sufficiently small absolute values the distance between $H(t_0)$ and $H(t_1)$  grows as
$|\sqrt{t_0} - \sqrt{t_1}|$ and in general cannot be controlled by $|t_0 - t_1|$. Indeed, if, for example, $t_0 = 0$ and  $t > 0$ is arbitrary small, then $H(0) = \{0\}$ (or one may write $H(0)=0$ to emphasize the single value) and
$d(H(0), H(t)) = \sqrt{t} \gg t$.

At last, we define a stronger notion of regularity, which will be useful later. Unlike lower Lipschitz semicontinuity, both $y$ and $y'$ are allowed to vary.
Here and throughout the paper, $B_\delta(y)$ denotes an open ball in $Y$ with center $y$ and radius $\delta$, and similarly $B_\epsilon(x)$ denotes a corresponding open ball in $X$.
\begin{definition}\label{def:aubin}
A map $F : Y \to 2^X$ is said to have the Aubin property at $(y_0, x_0) \in \textrm{Graph } F$ if there exists $\epsilon$, $\delta$ and $C > 0$ such that for all $y, y' \in B_\delta(y_0)$ and all $x \in B_\epsilon(x_0) \cap F(y)$
\[
d(x, F(y')) \leq C \norm{ y - y'}.
\]
\end{definition}

Note that we will use the Aubin property on the ``inverse'' of $F$ defined as
\[
    F^{-1}(y) = \{ x \mid x \in F(y) \}
\]
Note that the Aubin property of $F^{-1}$ implies a local bi-Lipschitz property of $F$.

\subsection{Sufficient Conditions for Calmness of Certain Solution Maps}
In the general setting of Banach spaces $X$, $Y$ (with the above notation) and $Z$, and functions $M : Y \to 2^X$ and $f : Z \times Y \to \R$, Klatte and Kummer \cite{Klatte2015} provide sufficient conditions for calmness at $(y_0, x_0) \in Y \times 2^X$ of
\begin{equation}\label{eq:argmin_mapping}
	S(y) =
	\argmin_{z \in M(y)} f(z, y).
\end{equation}
These sufficient conditions are requirements on the regularity of the constraint set $M(y)$ and the regularity of the mapping
\begin{equation}\label{eq:L_definition}
	L(y, \mu) = \{ x \in M(y) \mid f(x, y_0) \leq \mu \}
\end{equation}
for a certain choice of $\mu$. The regularity of these quantities is quantified by calmness and the following notion of lower Lipschitz semicontinuity.
It is almost identical to calmness except that $y$ varies instead of $x$.
\begin{definition}
A set-valued map $F: Y \to 2^X$ is called lower Lipschitz semicontinuous at $(y_0, x_0)$ if there exists $\delta, C > 0$ such that
\[
	d(x_0, F(y)) \leq C \norm{ y_0 - y_1}, \quad \textrm{ for all } y \in B_{\delta}(y_0).
\]
\end{definition}

We formulate Theorem 3.1 of Klatte and Kummer \cite{Klatte2015} on sufficient conditions for calmness of $S(y)$. We use the quantity $\mu_0 = f(S(y_0), y_0)$ for $y_0 \in Y$ and $x_0 \in X$. It follows from the definition of $f$ that this quantity is well-defined as $f$ is constant on all values of $S(y_0)$.
\begin{theorem}\label{theorem:argmin_calm}
Assume the general setting of Banach spaces $X$ (with metric $d$), $Y$ (with norm $\|\cdot\|$) and $Z$, and functions $M : Y \to 2^X$, $f : Z \times Y \to \R$, $S: Y \to 2^X$ defined in \eqref{eq:argmin_mapping} and $L: Y \times \R \to X$ defined in \eqref{eq:L_definition}.
If $y_0 \in Y$, $x_0 \in X$, $\mu_0 = f(S(y_0), y_0)$
and
\begin{enumerate}
	\item $M(y)$ is calm and lower Lipschitz semicontinuous at $(y_0, x_0)$;
	\item $L(y, \mu)$ is calm at $((y_0, \mu_0), x_0)$;
\end{enumerate}
then $S(y)$ is calm at $(y_0, x_0)$.
\end{theorem}

\subsection{Polyhedral Level Sets}
We assume that the unit ball of $\Omega$, $\{ \bx \mid \Omega(\bx) \leq 1 \}$, is a polyhedron.
For brevity, we refer by polyhedron to a convex polyhedron, which is the intersection of finitely many half spaces.

We relate the above property to the following well-known general notions of a polyhedral map and polyhedral level sets, which we use in this work. We define the former notion for a general function and the latter one only for $\Omega$.
For Banach spaces $X$ and $Y$, a map $F : X \to Y$ is a polyhedral map if there exists polyhedra $P_i, i = 1, \ldots, N$, such that
\[
	\textrm{Graph} F = \{ (x, F(x)) \mid x \in X\} = \bigcup_{i = 1}^N P_i.
\]
We say that $\Omega$ has polyhedral level sets if the level-set map
\begin{equation}
\label{eq:def_F}
	F(\mu) = \{ \bx \in \R^n \mid \Omega(\bx) \leq \mu \}
\end{equation}
is polyhedral. Since $\Omega$ is a seminorm, it is equivalent to the assumed property that $F(1)$ consists of a single polyhedron $P_1$ (this claim will be clearer from the geometric argument presented later in Section \ref{section:F_is_calm}).

We use a classic result from the theory of polyhedral mappings \cite{Jourani2000}:
\begin{theorem}\label{thm:polyhedral_map_theorem}
A polyhedral set-valued mapping $H : \R \to 2^{\R^n}$ is calm at all $(by_0, \bx_0) \in \textrm{Graph}(H)$.
\end{theorem}

\subsection{Principal Angles}
We exploit the geometric interplay between the affine subspace solution of $\bA \bx = \bb$ and the affine subspaces of the polyhedral level sets using principal angles between shifted linear subspaces \cite{Wedin1983}.  The principal angles $0 \leq \theta_1 \leq \cdots \leq \theta_r \leq \pi/2$
between two linear subspaces $\mathcal{U}$ and $\mathcal{V}$ of $\R^n$
with dimensions $r$ and $k$, respectively, where $r<k$, are defined recursively along with principal vectors $\bu_1$, $\ldots$, $\bu_r$ and $\bv_1$, $\ldots$, $\bv_r$. The smallest principal angle $\theta_1$ is given by
\begin{equation}\label{eq:theta1}
	\theta_1 = \min_{\substack{\bu \in \mathcal{U} ,\bv \in \mathcal{U}\\ \|\bu\| = \|\bv\| = 1 }} \arccos(|\bu^T \bv|).
\end{equation}
The principal vectors $\bu_1$ and $\bv_1$ achieve the minimum in \eqref{eq:theta1}.
For $2 \leq k \leq r$, the $k$th principal angle, $\theta_k$, and principal vectors, $\bu_k$ and $\bv_k$, are defined by
\begin{equation}\label{eq:thetak}
	\theta_k = \min_{\substack{\bu \in \mathcal{U}, \|\bu\| = 1, \bv \perp \bu_{1},\dots,\bu_{k-1} \\ \bv \in \mathcal{V}, \|\bv\| = 1, \bv \perp \bv_{1},\dots,\bv{k-1} }} \arccos(|\bu^T \bv|),
\end{equation}
where $\bu_k$ and $\bv_k$ achieve the minimum in \eqref{eq:thetak}

\subsection{Bounds for a System of Convex Inequalities}
\label{sec:bounds_system}
Our proof gives rise to a system of inequalities $f(\bx) \leq 0$, where $f$ is a convex function, and requires bounding the distance between the solution set of this system and a point that lies near the boundary of this set. The sharp constant of the required bound uses the subdifferential of $f$.

We first recall that the subdifferential of a convex function $f : \R^n \to \R$ is defined as the following set of subgradients of $f$:
\[
	\partial f(\bx) := \{ \bv \in \R^n \mid f(\bz) - f(\bx) \geq \innerp{\bv}{\bz - \bx} \textrm{ for all } \bz \in \R^n \}.
\]

The use of the subdifferential to bound distances to constraint sets began with Hoffman's original work \cite{Hoffman1952} on error bounds for systems of inequalities $\bA \bx \leq b$.
Define $\mathcal{S}_{\bA, \bb}$ to be the solution set of all $\bx \in \R^n$ that satisfy $\bA \bx \leq \bb$. Assuming this set is nonempty, Hoffman shows that there exists a constant $C$ so that for all $\bx \in \R^n$
\begin{equation}
    \label{eq:M_is_calm_bound_convex}
    d(\bx, \mathcal{S}_{\bA, \bb}) \leq C\|[\bA \bx - \bb]_+\|,
\end{equation}
where $([\bx]_+)_i = \max(x_i, 0)$ for $i = 1$, $\ldots$, $n$. In his original paper, he proved the bound using results from conic geometry and then computed the constants ad hoc for a few different norms. Later results, see e.g., \cite{NgaiHuynhAlexanderKruger2010}, generalize this to convex inequalities $f(\bx) \leq 0$ and compute sharp constants using the subdifferential $\partial f(\bx)$ . We will use the following theorem of \cite{NgaiHuynhAlexanderKruger2010}, where  \eqref{eq:local_error_bound} below is analogous to \eqref{eq:M_is_calm_bound_convex}.
\begin{theorem}\label{thm:local_error_bound}
	If $f : \R^n \to \R$ is a convex function, $Q = \{ \bz \in \R^n \mid f(\bz) \leq 0 \}$ and $\bx_0$ lies in the boundary of $Q$, then the following two statements are equivalent:
	\begin{enumerate}
	    \item
	There exist $c(f, \bx_0)$ and $\epsilon > 0$ such that for all $\bx \in B_{\epsilon}(\bx_0)$
	\begin{equation}\label{eq:local_error_bound}	
		d(\bx, Q) \leq c(f, \bx_0) [f(\bx)]_+;
	\end{equation}
	\item
	For all sequences $\{\bx_i\}_{i \in \mathbb{N}} \to \bx_0$ with $f(\bx_i) > 0$ for all $i \in \mathbb{N}$,
	\begin{equation}\label{eq:cond_error_bd}	
		\tau(f, \bx_0) := \liminf_{\bx_i \to \bx_0} d(0, \partial f(\bx)) > 0.
	\end{equation}
	\end{enumerate}
	Moreover $c = \tau(f, \bx_0)^{-1}$ is optimal in \eqref{eq:local_error_bound}.
\end{theorem}


\section{The Main Theorem and its Proof}\label{sec:proof}
We formulate the main theorem of this paper and follow up with its proof
\begin{theorem}
\label{theorem:main}
	Assume $m \leq n,$ $\bb \in \R^m$, $\bA \in \R^{m \times n}$ and $\Omega$ is a seminorm with polyhedral unit ball. Then the set-valued map defined in \eqref{eq:def_S}
	is calm at $(\bA_0, \bx_0)$ for every full-rank $\bA_0 \in \R^{m \times n}$ such that
	$\min_{\bA_0 \bx = \bb} \Omega(\bx) > 0$
	and for every $\bx_0 \in S(\bA_0)$.
\end{theorem}
In analogy to the notation used in Theorem \ref{theorem:argmin_calm}, we denote  $\mu_0 = \min_{\bA_0 \bx = \bb} \Omega(\bx)$. The assumption of Theorem \ref{theorem:main} that $\mu_0 > 0$ is reasonable from an applied perspective. For example, in the case of compressed sensing, where $\Omega(\bx) = \norm{\bx}_1$, $\mu_0 = 0$ corresponds to the trivial solution $\bx = \b0$. The condition $\mu_0 > 0$ thus only excludes the trivial solution, where the sparsity prior is meaningless.

In order to prove this theorem, we first prove in Section \ref{section:M_is_calm} that $M(\bA)$ is calm and lower Lipschitz at $(\bA_0, \bx_0)$. We then prove in Section \ref{sec:L_is_calm} that the map
\[
	L(\bA, \mu) = \{ \bx \in M(\bA) \mid \Omega(\bx) \leq \mu \}
\]
is calm at $((\bA_0, \mu_0), \bx_0)$.
In view of Theorem~\ref{theorem:argmin_calm}, these results imply Theorem \ref{theorem:main}.

\subsection{$M(\bA)$ is Calm and Lower Lipschitz at $(\bA_0, \bx_0)$}\label{section:M_is_calm}
The following classical property of the Moore-Penrose inverse will be useful. We offer a short proof here.
\begin{lemma}
\label{lemma:classic}
	Let $\bx \in \R^n$ and $\bA^\dagger$ be the Moore-Penrose inverse of $\bA$ \cite{campbell2009generalized}. Then the following bound holds
\begin{equation}\label{eq:psuedo_inverse_bound}
	d(\bx, M(\bA)) \leq \norm{\bA^\dagger} \norm{\bA \bx - \bb} .
\end{equation}

\end{lemma}

\begin{proof}
The operator $\bI - \bA^\dagger \bA$ is the projection onto the kernel of $\bA$ and thus the projection $P_{M(\bA)}$ is given by
\[
P_{M(\bA)} \bx = \bA^\dagger \bb + (I - \bA^\dagger \bA) \bx.
\]
Therefore,

\begin{equation*}
d(\bx, M(\bA)) =
\norm{\bx - P_{M(\bA)} \bx}
= \norm{\bx - \bA^\dagger \bb - \bx + \bA^\dagger \bA \bx}
= \norm{\bA^\dagger \bA \bx - \bA^\dagger \bb} \leq \norm{\bA^\dagger} \norm{\bA \bx - \bb}.
\end{equation*}

\end{proof}

We derive the calmness of $M(\bA)$ at $(\bA_0, \bx_0) \in \R^{m \times n} \times \R^{n}$ from the above lemma. We set the neighborhoods of Definition \ref{def:calm} as $U = \R^{m \times n}$ and $V$ any bounded neighborhood of $\bx_0$ in $\R^{n}$ of diameter $D>0$, where any choice of $D>0$ is fine, but it impacts the calmness constant (see below). We arbitrarily fix $\bA_1 \in U$ and $\bx_1 \in M(\bA_1) \cap V$. To verify calmness, we directly apply Lemma \ref{lemma:classic}
\[
d(\bx_1, M(\bA_0)) \leq
	\norm{\bA_0^\dagger}  \norm{\bA_0 \bx_1- \bb}
\]	
and further note that
\[
	\bA_0\bx_1 - \bb = \bA_0\bx_1 - \bb - \bA_1 \bx_1 + \bb = (\bA_0 - \bA_1)\bx_1.
\]
Combining the above two equations and using the triangle inequality yields the desired calmness of $M(\bA)$:
\begin{equation}\label{eq:M_is_calm_bound}
	d(\bx_1, M(\bA_0)) \leq
\norm{\bA_0^\dagger} \norm{\bx_1} \norm{\bA_0 - \bA_1}
\leq
\norm{\bA_0^\dagger} (\norm{\bx_0}+D) \norm{\bA_0 - \bA_1}.
\end{equation}
The calmness constant $\norm{\bA_0^\dagger} (\norm{\bx_0}+D)$  is bounded since the diameter $D$ is bounded and $\bA_0$ is full-rank, so its lowest singular value, $\sigma_\textrm{min}(\bA_0)$ is positive and thus $\norm{\bA^\dagger_0} = 1/{\sigma_\textrm{min}(\bA_0)}$ is bounded. We remark that the bound in the first inequality of \eqref{eq:M_is_calm_bound} is analogous to that in
\eqref{eq:M_is_calm_bound_convex}, though the latter one applies to a system of inequalities and not just equalities.

To see that $M(\bA)$ is lower Lipschitz semicontinuous, we use again the formula $\norm{\bA^\dagger} = 1/{\sigma_\textrm{min}(\bA)}$.
Since the complex roots of a polynomial vary continuously with respect to the coefficients,  the eigenvalues of a real symmetric matrix vary continuously with respect to that matrix. Thus $\sigma_\textrm{min}(\bA)$, which is the minimal eigenvalue of $\bA \bA^T$, depends continuously on $\bA$. Therefore, for every $\epsilon > 0$ there exists a $\delta > 0$ such that for $\bA_1 \in B_\delta(\bA_0)$, $\sigma_\textrm{min}(\bA_1)>0$, so that $\bA_1$ is also full-rank, and
\begin{equation}\label{eq:cony_of_const}
	\norm{\bA_1^\dagger} \leq (1 + \epsilon) \norm{\bA_0^\dagger}.
\end{equation}
Swapping $\bA_1$ and $\bA_0$ in the first inequality of \eqref{eq:M_is_calm_bound} and combining it with \eqref{eq:cony_of_const} gives
\[
	d(\bx_0, M(\bA_1)) \leq (1 + \epsilon) \norm{\bA_0^\dagger} \norm{\bx_0} \norm{\bA_0 - \bA_1}.
\]
This bound shows that $M(\bA)$ is lower Lipschitz semicontinuous at $(\bA_0, \bx_0)$.

\subsection{$L(\bA, \mu)$ is Calm at $((\bA_0, \mu_0), \bx_0)$}
\label{sec:L_is_calm}
Recall the map $F(\mu)$ defined in \eqref{eq:def_F} and note that we can write
\begin{equation}\label{eq:definition_of_L}
	L(\bA, \mu) = M(\bA) \cap F(\mu).
\end{equation}
Define also the system of inequalities
\[
    W(\mu) = M(\bA_0) \cap F(\mu).
\]
The following theorem from \cite{Klatte2015} gives conditions on $M$ and $F$ that ensure $L$ is calm.
\begin{theorem}
Let $X$, $Y$ ,$Z$ be Banach spaces. If $F_1 : Y \to 2^X$ is calm at $(y_0, x_0) \in \textrm{Graph }F_1$, $F_2 : Z \to 2^X$ is calm at $(z_0, x_0) \in \textrm{Graph } F_2$, $F_1^{-1}$ has the Aubin property at $(x_0, y_0)$. Finally, assume $F_3(z) := F_1(y_0) \cap F_2(z)$ is calm at $(z_0, x_0)$. Then $F_4(y, z) := F_1(y) \cap F_2(z)$ is calm at $((y_0, z_0), x_0)$.
\end{theorem}

We will apply the theorem with $X = \R^n$, $Y = \R^{m \times n}$, and $Z = \R$. We set $F_1 = M$, $F_2 = F$, $F_3 = W$, and $F_4 = L$. Let $((\bA_0, \mu_0), \bx_0) \in \textrm{Graph } L$. To see that $L$ is calm at this point, we thus verify that
\begin{enumerate}
	\item \label{item_1_to_prove} $F^{-1}(\bx)$ has the Aubin property at $(\bx_0, \mu_0)$.
	\item \label{item_2_to_prove} $F(\mu)$ is calm at $(\mu_0, \bx_0)$.
	\item \label{item_3_to_prove} $W(\mu)$ is calm at $(\mu_0, \bx_0)$.
\end{enumerate}
We verify property \ref{item_1_to_prove} in Section \ref{section:inverse_map_is_Aubin}, property \ref{item_2_to_prove} in Section \ref{section:W_is_calm},
and property \ref{item_3_to_prove}
in Section \ref{section:W_is_calm}.

The main trick is that we have reduced the problem from perturbations in $\bA$ to better understood cases of perturbations. For example, in verifying properties \ref{item_2_to_prove} and \ref{item_3_to_prove} we need to study tolerance to perturbations in
the right hand side of an inequality system, represented by either $F(\mu)$ or $W(\mu)$.
We prove the calmness of $W(\mu)$ and  $F(\mu)$ by both geometric or algebraic approaches. The geometric ones are quite simple, while the algebraic ones enable us to compute explicit constants.

\subsubsection{$F^{-1}$ is Aubin}\label{section:inverse_map_is_Aubin}
Since $\Omega$ is convex, it is Lipschitz on $B_\delta(\bx_0)$ for some $\delta > 0$. That is, there exists $C_\Omega \in \R$  so that
\begin{equation*}
	\abs{\Omega(\bx) - \Omega(\bx')} \leq C_\Omega \norm{\bx - \bx'} \textrm{ for all } \bx, \bx' \in B_\delta(\bx_0).
\end{equation*}
For $\bx$, $\bx' \in B_\delta(\bx_0)$ and  $\mu_1 \in F^{-1}(\bx) = [\Omega(\bx), \infty)$, we obtain by direct estimation and the above inequality the desired Aubin property as follows:
\[
	d(\mu_1, F^{-1}(\bx')) \leq \abs{\mu_1 - \Omega(\bx')} \leq \abs{\Omega(\bx) - \Omega(\bx')} \leq C_\Omega \norm{\bx - \bx'}.
\]

\subsubsection{$F$ is calm}\label{section:F_is_calm}
\noindent {\bf Geometric Proof:}
Note that by the degree-one homogeneity of the seminorm $\Omega$, for $\mu>0$
\[
	F(\mu) = \{\bx \in \R^n \mid \Omega(\bx) \leq \mu \} = \{\bx \in \R^n \mid \Omega(\bx / \mu) \leq 1 \}
	= \mu \cdot \{\bx \in \R^n \mid \Omega(\bx ) \leq 1 \}.
\]
That is, for $\mu > 0$, $F(\mu)$ is simply a scaling of the set $\{ \bx \mid \Omega(\bx) \leq 1 \}$.
By assumption, $F(1)$ is a polyhedron in $\R^n$. Since $F(\mu) = \mu F(1)$, then the graph of $F$ over any small interval is a polyhedron in $\R^{n+1}$. This observation and Theorem~\ref{thm:polyhedral_map_theorem} imply that $F$ is calm at $(\mu_0, \bx_0)$.

\noindent {\bf Algebraic Proof:}
This proof relies on Theorem~\ref{thm:local_error_bound} with  $f(\bx) := \Omega(\bx) - \mu_0$. We first verify \eqref{eq:cond_error_bd} with the latter $f$, which states that for all sequences $\{\bx_i\}_{i \in \mathbb{N}} \to \bx_0$ with $\Omega(\bx_i) > \mu_0$,
\begin{equation}\label{eq:calmness_constant_of_F}
		\tau(\Omega, \bx_0) := \liminf_{\bx_i \to \bx_0} d(0, \partial \Omega(\bx)) > 0.
\end{equation}
We note that if $\bx \in \R^n$ and $0 \in \partial \Omega(\bx)$, then $\Omega(\bx) = 0$. Indeed, $0 \in \partial \Omega(\bx)$ only when $\bx$ is a minimum of $\Omega$, and since $\Omega$ is a seminorm, its minimum is achieved only on the set $\Omega(\bx) = 0$. We further note that the assumption $\bx_0 \in F(\mu_0)$ implies that $\Omega(\bx_0) \geq \mu_0 > 0$. These two observations yield that
\begin{equation}\label{eq:distance_away_from_zero}
    d(0, \partial \Omega(\bx_0)) > 0.
\end{equation}
Next, we define $h(\bx) := d(0, \partial \Omega(\bx))$. If we can verify that $h$ is lower semicontinuous, then by \eqref{eq:distance_away_from_zero} and the definition of lower semicontinuity, we conclude \eqref{eq:calmness_constant_of_F} as follows:
$\liminf_{\bx_i \to \bx_0} h(\bx_i) \geq h(\bx_0) > 0.$
The lower-semicontinuity of $h$ becomes obvious when we rewrite it as
$h(\bx) = \inf_{ \by \in \partial \Omega(\bx)} \norm{\by}_2$. Indeed, it follows from the set-valued upper semicontinuity of $\partial \Omega(\bx)$ and the definitions of lower semicontinuity of a real-valued function and upper semicontinuity of a set-valued function.

Since we verified  \eqref{eq:cond_error_bd} with $f(\bx) = \Omega(\bx) - \mu_0$, Theorem~\ref{thm:local_error_bound}
implies that \eqref{eq:local_error_bound} holds with the same $f$ and also specifies the optimal constant in \eqref{eq:local_error_bound}. Note that by the definitions of $f$ and $F$, $Q = F(\mu_0)$. We thus proved that there exist $\epsilon > 0$, such that for all $\bx \in B_{\epsilon}(\bx_0)$
\begin{equation}\label{eq:g_local_error}
    d(\bx, F(\mu_0)) \leq \frac{1}{\tau(\Omega, \bx_0)} [\Omega(\bx) - \mu_0]_+.
\end{equation}
We conclude by showing that this bound implies the calmness of $F$. Let $\mu_1 \in B_\delta(\mu_0)$ (with $\delta$ chosen small enough so that $0 \notin B_\delta(\mu_0)$) and $\bx_1 \in F(\mu_1) \cap  B_{\epsilon}(\bx_0)$. By definition, $\Omega(\bx_1) \leq \mu_1$, and we can thus write
\begin{equation}\label{eq:Hoffman_trick}
	\Omega(\bx_1) - \mu_0 \leq \Omega(\bx_1) - \mu_0 - (\Omega(\bx_1) - \mu_1) = \mu_1 - \mu_0.
\end{equation}
The combination of \eqref{eq:g_local_error} and \eqref{eq:Hoffman_trick} results in
\begin{equation}
\label{eq:F_calmness}
    d(\bx_1, F(\mu_0)) \leq \frac{1}{\tau(\Omega, \bx_0)} [\mu_1 - \mu_0]_+ \leq \frac{1}{\tau(\Omega, \bx_0)} \abs{\mu_0 - \mu_1}
\end{equation}
for all $\bx_1 \in B_{\epsilon}(\bx_0) \cap F(\mu_1)$. Thus $F$ is calm in $(\mu_0, \bx_0)$.


\subsubsection{$W$ is calm}\label{section:W_is_calm}
\noindent {\bf Geometric Proof:}
Extending the polyhedral mapping proof from Section~\ref{section:F_is_calm} is nontrivial because $W$ does not obey a simple scaling law. Instead, we analyze carefully the intersection of the affine set, $M(\bA_0)$, and the polyhedral set, $F(\mu_0)$, near $\bx_0$.  We will bound the calmness constant for $W$ by $(\sin(\theta_{\min})\, \tau)^{-1}$, with $\tau = \tau(\Omega, \bx_0)$ from \eqref{eq:calmness_constant_of_F}, and $\theta_{\min} = \theta_{\min}(\Omega, \bx_0)$ the smallest positive principal angle between the affine space $M(\bA_0)$ and the affine sets defining the faces of $F(\mu_0)$ (of any dimension $<n$) containing $\bx_0$.


Denote by $\mathcal{U}_{1}$, $\ldots$, $\mathcal{U}_{p}$ the $p=p(\bx_0)$ faces of the polyhedron $F(\mu_0) \in \R^n$ of any dimension less than $n$ that contain $\bx_0$ (and thus also intersect the $m$-dimensional space $M(\bA_0)$).
Denote by $P_{F(\mu_0)}$ the projection operator of $\R^n$ onto $F(\mu_0)$. This operator is well-defined since $F(\mu_0)$ is convex. Choose $\epsilon > 0$ small enough so that the set
\[
    P_{F(\mu_0)}(W(\mu_1) \cap B_\epsilon(\bx_0))
\]
intersects only the faces $\mathcal{U}_{1}, \ldots, \mathcal{U}_{m}$ of $F(\mu_0)$. Clearly, it also intersects $M(\bA_0)$. Choose coordinates so that $\bx_0 = \b0$. Then $M(\bA_0)$, $\mathcal{U}_{1}, \ldots, \mathcal{U}_{m}$ can be viewed as linear subspaces of $\R^n$ and the following vectors in $\R^n$,
\[
\bx_0=\b0, \ \ \bx_1 \in W(\mu_1) \cap B_\epsilon(\bx_0),
\ \text{ and } \ \bx_2 := P_{F(\mu_0)}(\bx_1),
\]
form a right triangle, which is demonstrated in Figure~\ref{fig:triangle}.
\begin{figure}
\centering
\begin{tikzpicture}
    \centering
    \draw (0, 0) node[inner sep=0] {\includegraphics[width=.45\textwidth]{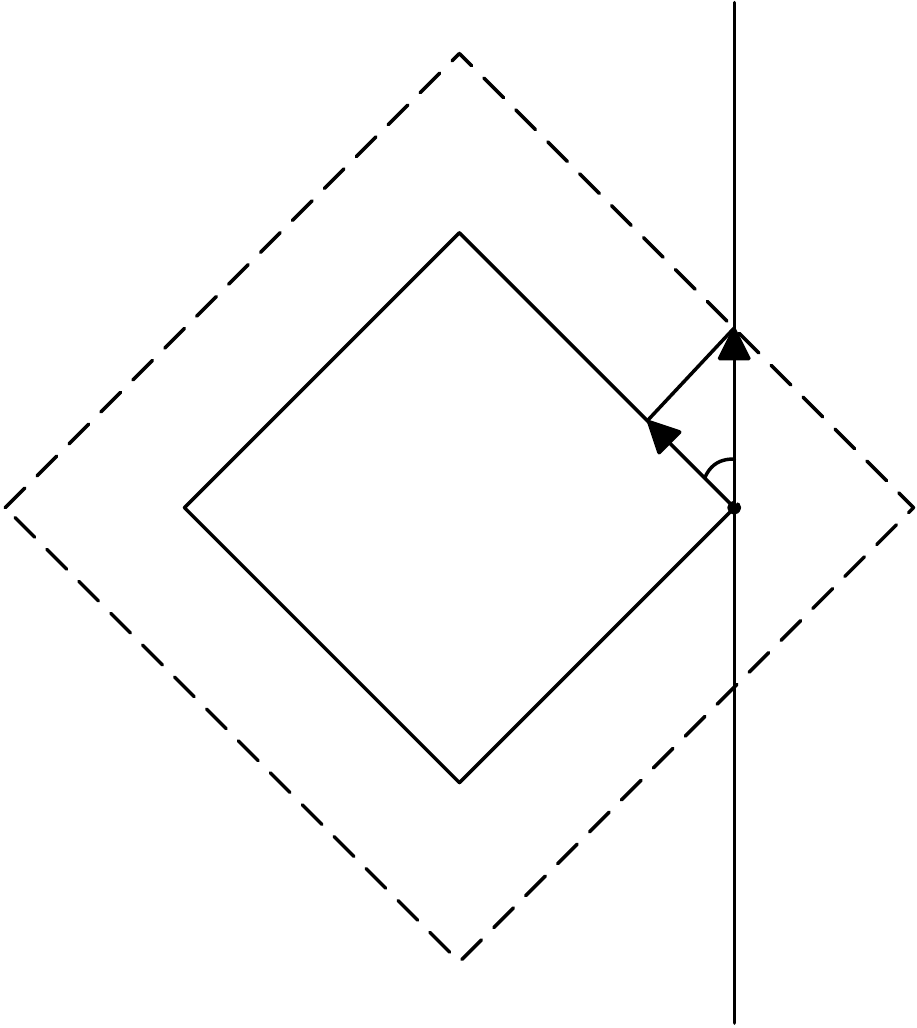}};
    \draw (3.0, 1.6) node {$\bx_1$};
    \draw (1.4, .8) node {$\bx_2$};
    \draw (3.0, 0) node {$\bx_0$};
    \draw (2.8, 5) node {$M(\bA_0)$};
    \draw (.2, 4.3) node {$F(\mu_1)$};
    \draw (.2, 2.8) node {$F(\mu_0)$};
    \draw (1.3, 1.6) node {$\mathcal{U}_1$};
    \draw (1.3, 1.6) node {$\mathcal{U}_1$};
    \draw (1.3, -1.6) node {$\mathcal{U}_2$};
    \draw (2.3, .7) node {$\theta$};
\end{tikzpicture}
\caption{Demonstration of the points $\bx_0$, $\bx_1$ and $\bx_2$ and related quantities.}
\label{fig:triangle}
\end{figure}
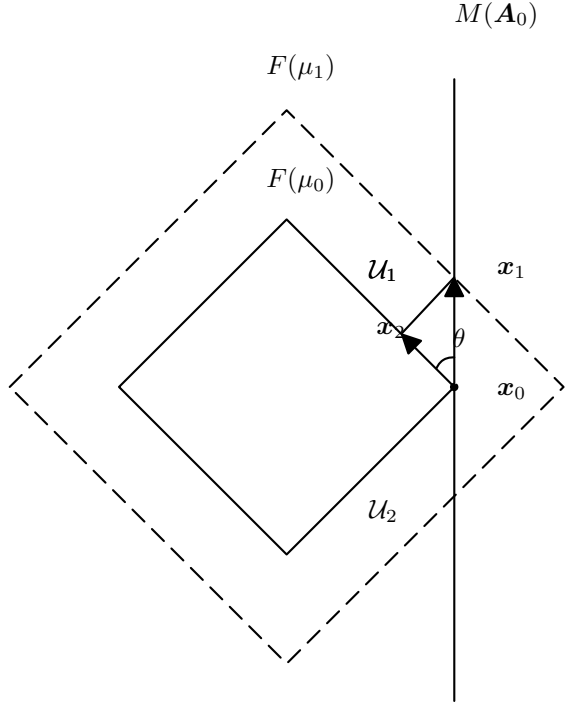
To prove that $W$ is calm at $(\mu_0, \bx_0)$ we wish to prove a bound of the form
\begin{equation}\label{eq:desired}
    d(\bx_1, W(\mu_0)) \leq c(\mu_0, \bx_0) \abs{\mu_0 - \mu_1}, \textrm{ for some } c(\mu_0, \bx_0) > 0.
\end{equation}

We start with the case when the triangle is degenerate, that is, either $\bx_1 = \bx_0$, or $\bx_2 = \bx_1$, or $\bx_2 = \bx_0$. First, assume that  $\bx_1 = \bx_0$. This implies that $\bx_1 \in W(\mu_0)$ and consequently the desired calmness: $d(\bx_1, W(\mu_0)) = 0 \leq \abs{\mu_0 - \mu_1}$.
Next, assume that $\bx_2 = \bx_1$. By definition, $P_{F(\mu_0)}(\bx_1) = \bx_1$ so $\bx_1 \in F(\mu_0)$. This observation and the fact that $\bx_1 \in M(\bA_0)$ implies that  $\bx_1 \in W(\mu_0)$ and the same estimate holds as above. At last, we assume that $\bx_2 = \bx_0$. Then $\bx_2 \in W(\mu_0)$ so
$d(\bx_1, W(\mu_0)) \leq \norm{\bx_1 - \bx_2}$
and we can bound $\norm{\bx_1 - \bx_2} = d(\bx_1, F(\mu_0))$ by \eqref{eq:F_calmness}.
We thus verified \eqref{eq:desired} for these cases with $c(\mu_0, \bx_0)=1/\tau(\mu_0, \bx_0)$.

In the case of non-degenerate triangles, we
note that since $\bx_0 =\b0 \in W(\mu_0)$, \begin{equation}\label{eq:distance_to_W}
    d(\bx_1, W(\mu_0)) \leq \norm{\bx_1 - \bx_0} = \norm{\bx_1}.
\end{equation}
Therefore, to obtain \eqref{eq:desired} we bound $\norm{\bx_1}$ in terms of $\abs{\mu_0 - \mu_1}$. Since $\b0$, $\bx_1$ and $\bx_2$ form a right triangle, $\norm{\bx_1} \sin \theta = \norm{\bx_1 - \bx_2}$,
where $\theta$ is the angle between $\bx_1$ and $\bx_2$. The combination of this observation with the calmness of $F$, stated in \eqref{eq:F_calmness}, results in the bound
\begin{equation}\label{eq:first_sin_equation}
    \norm{\bx_1} \sin \theta = \norm{\bx_1 - \bx_2} = d(\bx_1, F(\mu_0)) \leq \frac{1}{\tau(\Omega, \bx_0)} \abs{\mu_1 - \mu_0}.
\end{equation}
We note that $\theta$ depends on $\bx_1$ and thus \eqref{eq:first_sin_equation} does not yield yet the desired bound on $\norm{\bx_1}$.
To obtain this, we will establish a lower bound on $\theta$.

We index by $k=1$, $\ldots$, $\ell$, where $\ell \leq p$, the set of faces in $\mathcal{U}_1$, $\ldots$, $\mathcal{U}_p$ having at least one positive principal angle
with $M(\bA_0)$.
Note that since we ruled out the trivial case, where the triangle formed by $\bx_0$, $\bx_1$ and $\bx_2$ is degenerate, the latter set is not empty, that is, $\ell \geq 1$.
For $k=1$, $\ldots$, $\ell$, we denote
by $\theta_{k}^*$ the smallest positive principal angle between $\mathcal{U}_k$ and $M(\bA_0)$. Since $\bx_1$ and $\bx_2$ are nonzero and contained in distinct subspaces, they are orthogonal to the principal vectors corresponding to the zero principal angles (see \eqref{eq:thetak}). Hence $\theta_{k}^* \leq \theta$ for  $k=1$, $\ldots$, $\ell$. We further denote by $\theta_{\textrm{min}}$ the minimal value of $\theta_{k}^*$ among all $k=1$, $\ldots$, $\ell$ and have that  $\theta_{\textrm{min}} \leq \theta$.
Combining this observation with \eqref{eq:distance_to_W} yields the desired calmness of $W$ at $(\mu_0, \bx_0)$:
\[
    d(\bx_1, W(\mu_0)) \leq \frac{1}{\sin \theta_{k_\textrm{min}} \tau(\Omega \bx_0)}\abs{\mu_1 - \mu_u}.
\]

\noindent {\bf Algebraic Proof:}
The follows a similar approach to that in Section~\ref{section:F_is_calm}. We first write our set $W$ as the inequality system
\[
    W(\mu_0) = F(\mu_0) \cap M(\bA_0) = \{ \bx \in \R^n \mid f(\bx) \leq 0 \}, \quad \textrm{ where } f(\bx) := \max(\Omega(\bx) - \mu_0, d(\bx, M(\bA_0))).
\]
We note that by definition $f \geq 0$ and thus $W(\mu_0) =  \{ \bx \in \R^n \mid f(\bx) \leq 0 \}$, though we treat it as a system of inequalities.

Similarly to the proof in in Section~\ref{section:F_is_calm}, the calmness of $W(\mu)$ at $(\mu_0, \bx_0)$ with $\bx_0 \in W(\mu_0)$ follows from  \eqref{eq:cond_error_bd} with the later $f$. That is, we need to verify that for all sequences $\{\bx_i\}_{i \in \mathbb{N}} \to \bx_0$ with $f(\bx_i) > 0$
\begin{equation}
\label{eq:tau_last_one}
    	\tau'(f, \bx_0) := \liminf_{\bx_i \to \bx_0} d(0, \partial f(\bx)) > 0.
\end{equation}

The argument here is more subtle than in Section~\ref{section:F_is_calm} since here $0 \in \partial f(\bx_0)$, whereas before $0 \notin \partial \Omega(\bx_0)$.
We thus need a more subtle analysis of the function $\bx \mapsto d(0, \partial f(\bx))$. The definition of $W(\mu_0)$ and the fact that $f \geq 0$ imply that the minimal value of $f$, which is 0, is achieved at $W(\mu_0)$. Thus $0 \in \partial f(\bx)$ if and only if $f(\bx) = 0$.
Hence, for any sequence $\{\bx_i\}_{i \in \mathbb{N}} \to \bx_0$ with $f(\bx_i) > 0$, $0 \notin \partial f(\bx_i)$, or equivalently
\begin{equation}\label{eq:W_d_away_from_zero}
    d(0, \partial f(\bx_i)) > 0.
\end{equation}
In order to conclude \eqref{eq:tau_last_one}, we will show that $d(0, \partial f(\bx))$ is piecewise constant and combine this with \eqref{eq:W_d_away_from_zero}. By Theorem 3 of Chapter 4 of \cite{IoffeAleksandrDavidovich1979Toep} we have
\begin{equation}\label{eq:df}
	\partial f(\bx) = \textrm{co} \left( \{\partial \Omega(\bx)  \mid \Omega(\bx) - \mu_0 = f(\bx) \} \cup \{ \partial d(\bx, M(\bA_0)) \mid d(\bx, M(\bA_0)) = f(\bx)\} \right),
\end{equation}
where $\textrm{co} $ denotes the convex hull. Since $\Omega(\bx)$ has polyhedral level sets, $\partial \Omega(\bx)$ is piecewise constant in the sense that $\partial \Omega(\bx)$ outputs only finitely many different sets. Similarly, since $M(\bA_0)$ is also a polyhedron, $\partial d(\bx, M(\bA_0))$ is piecewise constant as a set-valued function. These two observations and \eqref{eq:df} thus imply that $\partial f(\bx)$ is piecewise constant as a set-valued function. Hence, $d(0, \partial f(\bx))$ is piecewise constant as a real-valued function and the proof is concluded.

\section{Conclusion}\label{section:future_work}

We have shown that if $\Omega$ has polyhedral level sets, then
the solution map $S(\bA)$ is calm.
By tracing the constants in our proof and in Klatte and Kummer \cite{Klatte2015}, we see that a calmness constant for $S$ is
\[
	\max \left(\norm{\bA_0^\dagger}, \frac{1}{\tau}\right)\left(1 +  \frac{2 C_\Omega}{\tau'}\right) C_\Omega \norm{\bA_0^\dagger},
\]
where $C_\Omega$ is the Lipschitz constant of $\Omega$ from Section~\ref{section:inverse_map_is_Aubin}, $\tau$ is defined in \eqref{eq:calmness_constant_of_F} and $\tau'$ is defined in \eqref{eq:tau_last_one}.

A key assumption for proving that $W$ is calm is that $\Omega$ has polyhedral sublevel sets. Removing this assumption would allow us to recover known calmness results such as the case $\Omega(\bx) = \norm{\bx}_2$ \cite{Ding1993}. Unfortunately, $W$ is in fact not calm for $\Omega(\bx) = \norm{\bx}_2$ as we show below. Therefore, our approach needs a major change if one wants to extend it to other known seminorms.

We verify our claim above by using the following equivalent definition of calmness: A map $G : Y \to 2^X$ is calm at $(\by_0, \bx_0)$ if and only if there exist a neighborhood $U$ of $\bx_0$ and a constant $c(\bx_0) > 0$ such that
\begin{equation}\label{eq:desired_bound}
	d(\bx, G(\by_0)) \leq c(\bx_0) \, d(\by_0, G^{-1}(\bx))
	\ \text{ for all } \ \bx \in U.
\end{equation}
Using this definition, we show that  $W(\mu)$ with $\Omega(\bx) = \norm{\bx}_2$ in $\R^2$, $\bA_0 = (0, -1)$ and $b = 1$ is not calm at $(\mu_0, \bx_0) = (1, (0, -1))$. Recall that by definition
\[
    W(\mu) = F(\mu) \cap M(\bA_0).
\]
In our case,
\[
    F(\mu) = \{ \bx \mid \norm{\bx}_2 \leq \mu \}, \quad M(\bA_0) = (0,-1)+\mathrm{Sp}((1,0)).
\]
Thus
\[
    W^{-1}(\bx) = [\norm{\bx}_2, \infty)
\]
so the distance is simply
\[
	d(\mu_0, W^{-1}(\bx)) = \norm{\bx}_2 - \mu_0 = \sqrt{\bx_1^2 + \bx_2^2} - \mu_0.
\]
Consequently, noting that $W(\mu_0) = (0, -1)$ we have
\[
		\lim_{\bx_2 = -1, \bx_1 \to 0} \frac{d(\mu_0, W^{-1}(\bx))}{d(\bx, W(\mu_0))} = \lim_{\bx_2 = -1, \bx_1 \to 0} \frac{\sqrt{\bx_1^2 + 1} - 1}{\bx_1} = 0
\]
and thus no such bound of the form \eqref{eq:desired_bound} exists with $c(\bx_0) > 0$.

\printbibliography

\end{document}